\documentclass[letterpaper, 10 pt, conference]{ieeeconf}  % Comment this line out if you need a4paper

\IEEEoverridecommandlockouts                              % This command is only needed if 
\usepackage{cite}
\usepackage{amsmath,amssymb,amsfonts}
\usepackage{algorithmic}
\usepackage{graphicx}
\usepackage{textcomp}
\usepackage{xcolor}
\newtheorem{definition}{Definition}
\newtheorem{theorem}{Theorem}
\newtheorem{lemma}{Lemma}

\newtheorem{example}{Example}
\def\BibTeX{{\rm B\kern-.05em{\sc i\kern-.025em b}\kern-.08em
    T\kern-.1667em\lower.7ex\hbox{E}\kern-.125emX}}
\DeclareGraphicsExtensions{.png,.pdf,.jpg}
\graphicspath{{./fig/}}
\begin{document}

\title{Weakly coupled systems
of eikonal equations in path-planning problems\\
}

\author{Maria Teresa Chiri$^{1}$,  Kenneth D Czuprynski$^{2}$ and Ludmil T Zikatanov$^3$% <-this % stops a space
\thanks{*The research of Ludmil Zikatanov is based upon work supported by and while serving at the National Science Foundation. Any opinion, findings, and conclusions or recommendations expressed in this material are those of the author and do not necessarily reflect the views of the National Science Foundation.}% <-this % stops a space
\thanks{$^{1}$ Department of Mathematics and Statistics,
        Queen's University, Kingston, ON, Canada
        {\tt\small maria.chiri@queensu.ca}}%
\thanks{$^{2}$ Applied Research Laboratory, The Pennsylvania State University,
        University Park, PA, USA
        {\tt\small kdc168@psu.edu}}%
        \thanks{$^{3}$ National Science Foundation,
        Alexandria, VA, USA
        {\tt\small  lzikatan@nsf.gov}}%
}

\maketitle

\begin{abstract}
In this paper, we study solutions for a weakly coupled system of eikonal equations arising in  an optimal path-planning problem with random breakdown. The model considered takes into account two types of breakdown for the vehicle, partial and total, which happen at a known, spatially inhomogeneous rate.  In particular, we analyze the complications due to the delicate degenerate coupling condition by using existing results on weakly coupled systems of Hamilton-Jacobi equations. Then we consider finite element method schemes built for convection-diffusion problems to construct approximate solutions for this system and produce some numerical simulations. 
\end{abstract}

\section{Introduction}
In the era of autonomous vehicles, there is an increasing need to develop robust and widely applicable path-planning algorithms. The models introduced so far are meant for a variety of autonomous vehicles which range from planetary exploration rovers to flying drones, underwater vehicles, and of course cars \cite{Reeds}.
 Although the literature on path planning methods used in robotics is already vast (see \cite{Alt, Kav, Lav} for some examples) and the number of models continues to increase over the years, there are still important issues that need to be addressed.  In the case of four-wheeled vehicles, for instance, part of the current research in path planning modeling aims to increase safety and performance by accurately estimating sideslip angles \cite{SELMANAJ20171, 6579735, 9400377}. On the other side, a new challenge, is to derive models taking into account accidental events such as vehicle breakdowns, even in the case of a completely known environment.
In this respect, in \cite{Cornell1} the authors consider an optimal path model in which a vehicle switches between two different modes, representing a partial and a total breakdown condition. The switch happens stochastically at known rates and each mode has its own deterministic dynamics and running cost. Their approach is based on dynamic programming in continuous state and time and leads to globally optimal
trajectories by solving a system of two weakly coupled eikonal equations. This represents a simple problem in the context of the general family of weakly coupled systems of Hamilton-Jacobi equations which appear frequently in the literature of optimal control problems with random switching costs governed by  Markov chains \cite{Dav, Flem, Prot}.  Recently weakly coupled Hamilton-Jacobi systems were studied from a viewpoint of weak KAM theory. For example, \cite{Camilli1} and \cite{Cagn} investigated the large-time behavior of the solution for time-dependent problems. In \cite{Mitake1}, the authors studied homogenization for weakly coupled systems and the rate of convergence to matched solutions while  \cite{Davini1} generalized the notion of Aubry sets for the case of systems and proved comparison principle with respect to their boundary data on Aubry sets.\\
 In this paper we are specifically interested in the system derived in \cite{Cornell1} and we will study it in the framework of viscosity solutions, as previously done in \cite{Engler, Ishi}. In particular, we will prove a comparison principle in the classical sense with the difference that the boundary conditions will be prescribed not only along the perimeter of the domain but also on the Aubry set of the system. On the other side, the existence of solutions is a very delicate issue. We will discuss an example of a one-dimensional weakly coupled system of eikonal equations showing how the lack of boundary conditions on the Aubry set affects the uniqueness of solutions and we will provide boundary conditions for which the system does not admit solutions.\\
Numerically, a range of different methods have been applied to solve the eikonal equations.
Fast marching methods \cite{Sethian1} and fast sweeping methods \cite{Kao, 10342800, parkinson2023efficient}  are efficient algorithms for the causal propagation of boundary information into the domain.
In \cite{Cornell1} the coupled system of eikonal equations is solved using  FMMs within a value iteration scheme which uses policy evaluations to speed up convergence.
A number of finite element formulations have also been considered for the eikonal equations (cf. \cite{Caboussat, Yang,guermond2008fast}).
In \cite{Yang}, an artificial viscosity regularization with a homotopy scheme for the stable reduction of the viscosity is introduced.
High-order discontinuous Galerkin formulations have also been developed for the eikonal equations \cite{Flad} and are amenable to parallelization.
In this work, we consider an artificial viscosity formulation similar to \cite{Yang} but consider finite element schemes built for convection-diffusion problems  \cite{Morton} as means of achieving stable numerical solutions. We point out that numerical simulations will be performed in a case in which the Aubry set is empty to avoid the theoretical complications described above. 
\smallskip

The paper is organized as follows: in Section II we will describe the model introduced in \cite{Cornell1}, in Section III we will prove a comparison principle for viscosity solutions of the weakly coupled system of eikonal equations and show an example of non-uniqueness and non-existence; in Section IV we will derive a numerical scheme based on a finite element method to approximate solutions for this system, while simulations will be discussed in section V. Finally in Section VI we will describe possible future research directions.

\section{The model}
Inspired by \cite{Cornell1}, we consider the problem of optimal path planning for an autonomous vehicle (AV) subject to two potential breakdown events: \\
1. \textit{\bf Partial breakdown} or \textit{\bf Mode 1}, in which the vehicle is damaged but can keep moving toward its target; \\
2. \textit{\bf Total breakdown} or \textit{\bf Mode 2}, in which the vehicle is damaged and needs to move toward a repair depot.\\
The classical formulation of the optimal path planning problem consists of an agent that tries to optimally travel from one point to another while obeying an ordinary differential equation describing its motion. 
In our analysis, we will consider AVs moving on a stretch of road modeled by a rectangle  $\Omega=[0,L]\times[0,S]$ where $L$ represents the length and $S$ is the width. The vehicles’ dynamic is described by the ODE
$$ \dot{y}(t)= d(t)f(y(t)), \quad y(0)=x_0$$
where $d:[0,T]\to S^1$ is  the direction of motion, $f$ is the speed and $y(t)$ is the trajectory starting at $x_0\in \Omega$.
Each breakdown type is assumed to occur at some known rate and accrues some predetermined cost.\\
Denote with $T_b $ the time at which the first partial breakdown occurs, $T_G$ the time at which the AV would reach the target if no breakdown occurred, and with $T_1=\min\lbrace T_b, T_G\rbrace $, then the expected cost in Mode 1 is given by 
\begin{align}\label{C1} J_1(x, a(\cdot))=\mathbb{E}\Big[&\int_0^{T_1}K_1(y(s))+\lambda(y(s))R(y(s))ds \notag \\
&+\chi_{[T_b<T_G]}u_2(y(T_b)) \Big]\,, \end{align}
where $u_2$ is the value function in Mode 2, $K_1$ is the running cost, $\lambda>0$ is the rate at which the partial breakdown occurs (trajectory dependent) and $R$ is the repair cost.
If a partial breakdown occurs before the robot reaches the goal, then the last term in \eqref{C1} is nonzero and the AV switches to Mode 2. 
Similarly, let $T_B$ be the time of the first total breakdown, $T_D$ the time the AV would reach the depot if no breakdown occurred, and denote with $T_2=\min\lbrace T_B, T_D \rbrace$, then the expected cost in Mode 2 is given by 
\begin{align}\label{C2} J_2(x, a(\cdot))=\mathbb{E}\Big[&\int_0^{T_2}K_2(y(s))ds \notag \\
&+\chi_{[T_B<T_D]}\Big( R(y(T_B))+u_1(y(T_B))\Big)\notag\\
&+\chi_{[T_B\geq T_D]}R(y(T_D))+u_1(y(T_D)) \Big]\,,\end{align}
where $K_2$ is the running cost.
The last two terms in \eqref{C2} again encode the optimal cost-to-go whenever a mode switch occurs.
Let $G \subset \Omega$ denote the goal locations for admissible AV trajectories and let $D \subset \Omega$ denote the location of the repair depots that the AV can navigate to after undergoing a partial breakdown. According to the derivation described in \cite{Cartee}, the value functions $u_1$ and $u_2$ solve the following system of eikonal equations:
\begin{equation}\label{eikonal-system-original}
\begin{cases}
|\nabla u_1| f_1 + \phi_1 (u_1-u_2)
	= K_1 + \lambda R, 
\quad &x\in \Omega \\
|\nabla u_2| f_2 + \phi_2(u_2 -u_1)
	= K_2 + \phi_2 R, \quad &x \in \Omega\\
u_1 =0, \quad &x \in G\\
u_2=R+ u_1, \quad &x \in D
\end{cases}
\end{equation}
Observe that $i$-th equation in the system does not depend on $\nabla u_j$ for $i\neq j$; therefore we refer to $\eqref{eikonal-system-original}$ as a weakly coupled system.
The functions $\lambda, \phi_2:\Omega \rightarrow (0,\infty)$ correspond to the rates of exponential random variables modeling the time until another complete breakdown and they are trajectory dependent; the function $\phi_1:\Omega \rightarrow (0,\infty)$ denotes the rate at which partial breakdowns occur. We note that providing a consistent notation to the coupling matrix necessarily leads to an asymmetry in the notation for the rates of the exponential random variables and the rate for the partial breakdowns.
The functions $K_i:\Omega \rightarrow [0,+\infty)$ denote the running cost associated with the cost functional, while $f_i:\Omega\to [0,+\infty) $ are the space dependent speeds.

%All of these functions represent the specific model. 
%However, in total, they are simply functions in the formulation that are assumed known.

\section{Viscosity solutions and comparison principle}
In this section, we want to prove a comparison principle for viscosity solutions of the weakly coupled system \eqref{eikonal-system-original}. We start by recalling the definition of viscosity solution introduced in \cite{Visc}, in the context of the system we are considering. 

\smallskip
\begin{definition}A continuous function $u:\Omega\to\mathbb{R}^2$, with $u(x)=(u_1(x),u_2(x))$,
is a {\it viscosity subsolution} (resp. {\it supersolution}) of \eqref{eikonal-system-original}, if for each $i=1,2$ and test function $\varphi$ of class $\mathcal{C}^1$, when $u_i-\varphi$ attains its local maximum (resp. local minimum) at $x$, then 
\begin{equation}\label{ESV}
\begin{aligned}
|\nabla \varphi | f_1+ \phi_1u_1 
	&\leq K_1 + \lambda R + \phi_1 u_2, \\
|\nabla\varphi| f_2 + \phi_2u_2  
	&\leq K_2 + \phi_2 R+ \phi_2u_1, 
\end{aligned}\quad (\hbox{resp.}~\geq~)
\end{equation}
A function is called a {\it viscosity solution} of \eqref{eikonal-system-original} if it is both a viscosity subsolution and supersolution.
\end{definition}
\medskip
A standard assumption in literature for weakly coupled systems of Hamilton-Jacobi equations is the so-called monotonicity condition, i.e. the $i$-th Hamiltonian is increasing in $u_i$ and decreasing in $u_j$ with $j\neq i$. Such a condition is clearly satisfied by \eqref{eikonal-system-original}. Moreover, the Hamiltonians are locally Lipschitz continuous, coercive, and convex in the gradient of $u$. 
On the other side, the coupling matrix for the system is given by 
\begin{equation}\label{coup} D(x)=\begin{pmatrix} \phi_1 & -\phi_1\\ -\phi_2 & \phi_2  \end{pmatrix} \end{equation}which is degenerate (the sum of the element of each row is equal to zero) and irreducible for every $x\in\Omega$, this simply means that coupling is not trivial. Degenerate couplings have been avoided for a long time, since they represented an obstruction to achieving uniqueness of the solution (cf. \cite{Engler}). This was until it was realized this type of coupling induces pathologies on the Aubry set which behaves as a boundary in the inner part of the domain \cite{Roquejoffre, Camilli1} . 
This set, in the particular case of \eqref{eikonal-system-original}, is given by
\begin{align*}
    &\mathcal{A}\doteq\lbrace  x\in\Omega, K_1(x)+\lambda(x)R(x)=0 \rbrace\\&\qquad\qquad\cap \lbrace  x\in\Omega, K_2(x)+\phi_2(x)R(x)=0 \rbrace \end{align*}

 and we may have  $\mathcal{A}\setminus \partial \Omega \neq \emptyset$.
Therefore to prove a comparison principle and consequently to get the uniqueness of the solution, it is essential to add boundary conditions on $\mathcal{A}\cup \partial \Omega$.
 In the following, we will use the notation $USC(\Omega)$ and $LSC(\Omega)$ for the set of upper semicontinuous and lower semicontinuous functions on $\Omega$.

\smallskip
\begin{theorem}\label{CP}[Comparison Principle]
Let $\phi_1, \phi_2, \lambda, R:\Omega\to [0,\infty)$ be continuous functions. Let $u\in USC(\Omega)$ and $v\in LSC(\Omega)$  be respectively bounded subsolutions and supersolution of \eqref{eikonal-system-original}.
Assume that $u(x)\leq v(x)$ for every $x\in \mathcal{A}\cup \partial \Omega$, then $u(x)\leq v(x)$ on $\bar\Omega$

\end{theorem}

\begin{proof}
Consider $\delta\in(0,1)$ an set
$$M_\delta\doteq \sup_{[0,L]\times[0,S]}\sup_{i=1,2}\lbrace \delta u_i-v_i \rbrace $$

If $M_\delta\leq 0$ there is nothing to prove, therefore we assume it is strictly positive. 
By compactness the $sup$ is a maximum achieved at a point $x_0\in[0,L]\times[0,S]$ for some $i_0$.
Denote with $\mathcal{I}$ the set of the index for which the maximum is achieved. We can distinguish three cases.\\
\textbf{Case 1.} $\mathcal{I}=\lbrace 1, 2\rbrace$ and $x_0\in\mathcal{A}\cup \partial \Omega$. In this case we have 
\begin{align*}
M_\delta
& =\frac{1}{2}\big[(\delta u_1-v_1)(x_0)+(\delta u_2-v_2)(x_0)\big]\\
&\leq\frac{1}{2}(\delta-1)\big(v_1(x_0)+v_2(x_0)\big)\\
&\leq(1-\delta)|v|_\infty.
\end{align*}
\textbf{Case 2.} $\mathcal{I}=\lbrace 1, 2\rbrace$ and $x_0\notin\mathcal{A}\cup\partial \Omega$. For instance, we can assume without loss of generality that $K_1(x_0)+\lambda(x_0)R(x_0)\neq 0$. In this case, we proceed classically by considering 
$$ \sup_{\Omega\times\Omega}\Big\{ \delta u_1(x)-v_1(y)-\frac{|x-y|^2}{2\varepsilon^2}-|x-x_0|^2\Big \}. $$
Let $(\bar{x},\bar{y})$ be the point at which this maximum is achieved. This value is clearly greater than $M_\delta$, moreover, the following properties are satisfied:
$$ \bar{x}, \,\bar{y} \to x_0\,\qquad \frac{|x-y|^2}{2\varepsilon^2},\,|\bar{x}-x_0|^2\to 0\quad \hbox{ as }\varepsilon\to 0.  $$
Let $p_\varepsilon=\frac{\bar{x}-\bar{y}}{\varepsilon^2}$. Since $u_1$ is viscosity subsolution of \eqref{eikonal-system-original}, we get 
\begin{align*}\Big\vert {p_\varepsilon+2 (\bar{x} - x_0)}\Big \vert f_1(\bar{x})+\delta\phi_1(\bar{x})&(u_1(\bar{x})-u_2(\bar{x}))\\ 
&\leq \delta \big(K_1(\bar{x})+\lambda(\bar{x})R(\bar{x})\big). \end{align*}
On the other side,  since $v_1$ is viscosity supersolution 
$$ \vert p_\varepsilon\vert f_1(\bar{y} )+\phi_1(\bar{y})(v_1(\bar{y})-v_2(\bar{y}))\geq\big(K_1(\bar{y})+\lambda(\bar{y})R(\bar{y})\big).$$
By subtracting the two inequalities we first get
\begin{align*} &\Big\vert {p_\varepsilon+2 (\bar{x} - x_0)}\Big \vert  f_1(\bar{x})- \vert p_\varepsilon\vert  f_1(\bar{y} )\\
&=\Big(\Big\vert {p_\varepsilon+2 (\bar{x} - x_0)}\Big \vert -\vert p_\varepsilon\vert\Big)  f_1(\bar{x} )+\vert p_\varepsilon\vert \big( f_1(\bar{x} ) - f_1(\bar{y} )\big)
  \end{align*}
which goes to $0$ as $\varepsilon\to 0$, and
\begin{align*}\delta\phi_1(\bar{x})&\big(u_1(\bar{x})-u_2(\bar{x})\big)-\phi_1(\bar{y})\big(v_1(\bar{y})-v_2(\bar{y})\big)\\
&=\phi_1(\bar{x})\big(\delta u_1(\bar{x})-\delta u_2(\bar{x})-v_1(\bar{y})+v_2(\bar{y})\big)\\
&\qquad \qquad +(\phi_1(\bar{x})-\phi_1(\bar{y}))(v_1(\bar{y})-v_2(\bar{y}))
\end{align*}
where 
$$(\phi_1(\bar{x})-\phi_1(\bar{y}))(v_1(\bar{y})-v_2(\bar{y}))\to 0 \qquad \hbox{as }\quad \varepsilon\to 0. $$
Therefore we get
\begin{align*} \phi_1(\bar{x})&\big(\delta u_1(\bar{x})-\delta u_2(\bar{x})-v_1(\bar{y})+v_2(\bar{y})\big)\\
&\leq (\delta-1)\big(K_1(\bar{x})+\lambda(\bar{x})R(\bar{x})\big)+o_\varepsilon(1).\end{align*}
Since $\delta u_j-v_j$ is upper-semicontinuous for $j=1,2$, it follows that
\begin{equation}\label{diff} \limsup_{\varepsilon\to0}(\delta u_j(\bar{x})-v_j(\bar{y}) )\leq \delta u_j(x_0)-v_j(x_0)\leq M_\delta. \end{equation}
hence
$$-\phi_1(\bar{x})(\delta u_2(\bar{x})-v_2(\bar{y}))\geq -\phi_1(\bar{x})M_\delta+o_\varepsilon(1) .$$
Moreover we have $\delta u_1(\bar{x})-v_1(\bar{y})\geq M_\delta$,
therefore
$$ \phi_1(\bar{x})(\delta u_1(\bar{x})-v_1(\bar{y}))\geq\phi_1(\bar{x})M_\delta. $$

We can conclude that 
$$ (\delta-1)\big(K_1(\bar{x})+\lambda(\bar{x})R(\bar{x})\big)\geq 0$$

which leads to contraddiction when $\varepsilon\to 0$ since $x_0\notin\mathcal{A}$.

\textbf{Case 3.} $\mathcal{I}\neq\lbrace 1, 2\rbrace$.
We can argue almost like in Case 2, but the estimate needs to be more rigorous on the index not in $\mathcal{I}$. Indeed let's assume that $\mathcal{I}=\lbrace{1\rbrace}$. 
The previous computations still hold, but for \eqref{diff}, if $j=2$ 
$$\limsup_{\varepsilon\to 0} (\delta u_2(\bar{x})-v_2(\bar{y}) )\leq \delta u_2(x_0)-v_2(x_0)\leq M_\delta-\eta$$
for some $\eta>0$. It follows that 
$$\phi_1(x_0)\eta\leq (\delta-1)\big(K_1(\bar{x})+\lambda(\bar{x})R(\bar{x})\big)+o_\varepsilon(1)$$
which again leads to contradiction as $\varepsilon\to 0$.

We can conclude that the only possibility is $M_\delta \leq(1-\delta)|v|_\infty $ which in the case of $\delta=1$ implies $M_1\leq0$.
\end{proof}
\medskip

\begin{lemma} Let $u\in USC(\Omega)$ be a bounded viscosity subsolution of \eqref{eikonal-system-original}, then $u$ is Lipschitz continuous on $\Omega$.
\end{lemma}
\begin{proof} Since $\Omega$ is a compact set, we have that $K_1, K_2, \phi_1,\phi_2, \lambda$ are bounded.  Therefore
$$ |\nabla u_i|f_i(x)\leq C_i \qquad \hbox{i=1,2}$$
which implies that $u$ is Lipschitz. 
\end{proof}
\medskip
The comparison principle proved in Theorem \ref{CP}  implies that if a solution exists, it must be unique. On the other hand, it is not clear which type of boundary conditions on $\partial\Omega$ and $\mathcal{A}$ ensure the existence of viscosity solutions. We borrow an example from \cite{Camilli1} to illustrate how the set $\mathcal{A}$ can affect the uniqueness of solutions for a weekly coupled system. Moreover, we provide bundary conditions on $\mathcal{A}$  for which such a system has no solution.\\
\begin{example} Consider the one-dimensional problem  
\begin{equation}\label{EX}
\begin{cases}
|u'_1(x)|+u_1(x)-u_2(x)=F(x),\quad &x\in(-1,1)\\
|u'_2(x)|+u_2(x)-u_1(x)=F(x)\quad &x\in(-1,1)\\
u_i(\pm 1)=0\quad & i=1,2
\end{cases}
\end{equation}
with $F(x)=2|x|$. It is immediate to check that both $u_1(x)=u_2(x)= 1-x^2$ and $u_1(x)=u_2(x)=\min\lbrace 1-x^2, x^2-C\rbrace $, $C\in(0,1)$, are viscosity solutions of \eqref{EX}. Here the set $\mathcal{A}$ is given by $\lbrace 0\rbrace$,  hence without prescribing a boundary condition in $0$ we do not have uniqueness. If we just add the condition $u_i(0)=\frac{1}{2}$, this immediately selects the unique solution $u_1(x)=u_2(x)=\min\lbrace 1-x^2, x^2-\frac{1}{2}\rbrace$. On the other side for $u_i(0)=2$, no solution satisfies \eqref{EX}.
\end{example}

\section{Numerical Solution}
In this section the numerical solution of the weakly coupled system \eqref{eikonal-system-original} using the finite element method (FEM) is considered.
%This is in contrast to the fast marching / fast sweeping schemes commonly used (cf. \cite{Seithian96, Zhao2005}.
The numerical scheme will solve a sequence of systems augmented by artificial viscosity terms.
Formally, we have the problem: Find $u^\varepsilon_1$ and $u^\varepsilon_2$ such that
\begin{equation}\label{eikonal-system-augmented}
\begin{aligned}
-\varepsilon \Delta u_1^\varepsilon 
+ |\nabla u^\varepsilon_1| f_1 
+ \phi_1 u^\varepsilon_1
&= K_1 + \lambda R + \phi_1u^\varepsilon_2,  \\
-\varepsilon \Delta u^\varepsilon_2 +  |\nabla u^\varepsilon_2| f_2 + \phi_2u^\varepsilon_2 
	&= K_2 + \phi_2 R+ \phi_2u^\varepsilon_1,
\end{aligned}
\end{equation}
for $x \in \Omega$ with $u^{\varepsilon}_1 = 0$ for $x \in G$ and $u^{\varepsilon}_2 = R + u^{\varepsilon}_1$ for $x \in D$, as $\varepsilon \rightarrow 0$.
Although the viscosity term provides some mechanism for control of the solution, we are still left to deal with the original nonlinearity which is accomplished through linearization as in \cite{Yang}.
%Recalling that we solve \eqref{eikonal-system-augmented} for a sequence of $\varepsilon \rightarrow 0$,
Let  $u^{\varepsilon_n}_k$ denote the solution of \eqref{eikonal-system-augmented} for the $n^{th}$ value of $\varepsilon$, where $k=1,2$.
Linearization of the absolute value term results in the approximation 
\begin{align} \label{eqn:linearization}
|\nabla u^{\varepsilon_n}_{k}| 
& \approx |\nabla u^{\varepsilon_{n-1}}_{k}| 
+\frac{\nabla u^{\varepsilon_{n-1}}_{k}}{|\nabla u^{\varepsilon_{n-1}}_{k}|} \left(\nabla u^{\varepsilon_{n}}_{k}- \nabla u^{\varepsilon_{n-1}}_{k}\right) \nonumber \\
&\approx \frac{\nabla u^{\varepsilon_{n-1}}_{k} }{|\nabla u^{\varepsilon_{n-1}}_{k}|} \nabla u^{\varepsilon_n}_{k},
\end{align}
for $\nabla u ^{\varepsilon_{n-1}}_{k}$ not identically zero.
Inserting the approximation \eqref{eqn:linearization} into \eqref{eikonal-system-augmented} leads to the following iteration. Let $u^{\varepsilon_0}_1$ and $u^{\varepsilon_0}_2$ denote initial guesses and let $\beta^{\varepsilon_{n-1}}_k$ denote the normalization of $\nabla u^{\varepsilon_{n-1}}_{k}$, then for $n = 1,2,\dots$, we solve
\begin{equation}\label{eikonal-system-w-linearization}
\begin{aligned}
-\varepsilon_n \Delta u^{\varepsilon_n}_1
+ \beta^{\varepsilon_{n-1}}_1\cdot \nabla u^{\varepsilon_n}_1 f_1 
+ \phi_1 u^{\varepsilon_n}_1
&= K_1 + \lambda R + \phi_1u^{\varepsilon_n}_2  \\
-\varepsilon_n \Delta u^{\varepsilon_n}_2 
+  \beta^{\varepsilon_{n-1}}_2 \cdot\nabla u^{\varepsilon_n}_2 f_2 + \phi_2u^{\varepsilon_n}_2 
	&= K_2 + \phi_2 R+ \phi_2u^{\varepsilon_n}_1 
\end{aligned}
\end{equation}
for $x \in \Omega$ with $u^{\varepsilon_n}_1 = 0$ for $x \in G$ and $u^{\varepsilon_n}_2 = R + u^{\varepsilon_n}_1$ for $x \in D$.
Equation \eqref{eikonal-system-w-linearization} now resembles a convection-diffusion problem for which a number of numerical schemes exist.
%When considering finite element methods, techniques such as the streamline diffusion approximation  \cite{Johnson} or EAFE \cite{} can be employed.
In this work, we formulate the discrete system via a streamline diffusion approximation.

%%%%%%%%%%%%%%%%%%%%%%%%%%%%%%%%
% Streamline Diffusion
%%%%%%%%%%%%%%%%%%%%%%%%%%%%%%%%
\subsection{Streamline Diffusion}
For the discretization of equation \eqref{eikonal-system-w-linearization} we use a streamline diffusion approximation~\cite{Brooks1,Johnson} and in our derivation we follow the settings in~\cite{Bank1}.
We seek an approximate solution which is a piecewise linear, continuous function on a given simplicial mesh $\mathcal{T}_h$   covering exactly the computational domain $\Omega$, that is, 
$\overline{\Omega} = \bigcup_{T\in\mathcal{T}_h} T$. We also assume that the sets $G$ and $D$ are partitioned by the mesh exactly. Then, for $\Gamma$ denoting either of the sets $G$ or $D$, the piecewise linear continuous finite element space is given by introducing
\begin{align*}
V_{\Gamma;h} :=
\left\lbrace
v \in C(\overline{\Omega}) \; | \;
v|_T \in P_1(T), \; \forall T \in \mathcal{T}_h, 
v|_{\Gamma} = 0
\right\rbrace
\end{align*}
where $P_1(T)$ denotes the space of polynomials of degree $\le 1$ over $T$. For a given piecewise linear $g$ which takes prescribed values on $\Gamma$ the numerical solutions lie in the affine spaces $V_\Gamma(h)=g+V_{h;\Gamma}$, $\Gamma=D$ or $\Gamma=G$.  Thus, the solution spaces for $u_1$ and $u_2$ differ slightly and are given by $V_G(h)$ and $V_D(h)$, respectively.

For ease of presentation in the derivation of the weak form of \eqref{eikonal-system-w-linearization}, we suppress the superscript notation as well as the iteration variable $n$.
The test functions for the streamline diffusion formulation have the form $v + \theta h \nu\beta\cdot \nabla v$ for $v \in V_\Gamma(h)$, where $\nu = (|\beta|^2 +1)^{-1/2}$. 
The variable $h$ denotes the mesh size and $\theta$ is a positive constant.
Consider the first equation in \eqref{eikonal-system-w-linearization}, then for $v_1 \in V_D(h)$ multiplying by a test function and integrating yields
\begin{equation}\label{eqn:weak-form-derivation-1}
\begin{aligned}
\left(-\varepsilon\Delta u_1 + \beta_1 \cdot \nabla u_1 f_1 
+ \phi_1 u_1, v_1 + \theta h \nu\beta_1\cdot \nabla v_1\right)&  \\
&\hspace*{-2.5in}= 
\left(\phi_1u_2, v_1 + \theta h \nu\beta_1\cdot \nabla v_1\right)\\ 
&\hspace*{-2in}+
\left(K_1 + \lambda R, v_1 + \theta h \nu\beta_1\cdot \nabla v_1\right)
\end{aligned}
\end{equation}
where $(u,v) = \int_{\Omega} uv d\omega$.
We remark that because $v \in V_\Gamma(h)$, the components of $\nabla v$ are constant on each element and therefore integrations involved in calculating the form $(w, \theta h \nu\beta\cdot \nabla v)$ are done elementwise.
Working from right to left, define the linear functional $l_1(\cdot)$ by
\begin{align} 
\label{eqn:weak-form-derivation-2}
l_1(v_1) := \left(K_1 + \lambda R, v_1 + \theta h \nu\beta_1\cdot \nabla v_1\right)\end{align}
and the bilinear form $b_1(\cdot, \cdot)$ by
\begin{align} 
\label{eqn:weak-form-derivation-3}
b_1(u_2,v_1) := \left(\phi_1u_2, v_1 + \theta h \nu\beta_1\cdot \nabla v_1\right).
\end{align}
The final term can be split into the sum of the viscosity term and its remainder.
Then for $\beta_1 = ([\beta_1]_1, [\beta_1]_2)^T$ the  introduction of the matrix
\begin{align*}
{\bf B}_1 = \begin{pmatrix}
[\beta_1]_1^2 & [\beta_1]_1[\beta_1]_2 \\
[\beta_1]_1[\beta_1]_2&   [\beta_1]_2^2
\end{pmatrix}f_1
\end{align*}
allows us to represent the next term as
\begin{equation}\label{eqn:weak-form-derivation-4}
\begin{aligned}
\left(\beta_1 \cdot \nabla u_1 f_1 
+ \phi_1 u_1, v_1 + \theta h \nu\beta_1\cdot \nabla v_1\right)&\\
%&=
%\theta h^p \nu \left(\beta_1 \nabla u_1 f_1 ,\beta_1\cdot \nabla v_1\right)
%+
%\left(\beta_1 \cdot\nabla u_1 f_1 , v_1\right)
%+
%\left( \phi_1 u_1, v_1 + \theta h \nu\beta_1\cdot \nabla v_1\right) \nonumber \\
&\hspace{-2in}=
\theta h \nu \left(\nabla u_1 ,{\bf B}_1 \nabla v_1\right)
+
\left(\beta_1\cdot \nabla u_1 f_1 , v_1\right)\\
+
\left( \phi_1 u_1, v_1 + \theta h \nu\beta_1\cdot \nabla v_1\right).
\end{aligned}
\end{equation}
Performing integration by parts on the viscosity term yields
\begin{equation}\label{eqn:weak-form-derivation-5}
\begin{aligned}
-\varepsilon \left( \Delta u_1, v_1 + \theta h \nu\beta_1\cdot \nabla v_1\right)&
\\&\hspace*{-1.5in}=
\varepsilon(\nabla u_1, \nabla v_1) - \varepsilon\left(\nabla u_1, v_1+ \theta h^p \nu\beta_1\cdot \nabla v_1\right)_{\partial \Omega }
\end{aligned}
\end{equation}
where $(\nabla u , v)_\Gamma = \int_{\Gamma} v \nabla u \cdot n dS$.
We note that due to the definition of $V_{\Gamma}(h)$, all higher derivatives of $\nabla v_1$ and $\beta_1$ are zero for any given element.
In this case, integration by parts results in an additional term due to the set $D$ not necessarily being the entire (or part) of the boundary of the domain.
By letting ${\bf K}_1 = {\bf I} \varepsilon + \theta h \nu  {\bf B}_1$
we can define the bilinear form $a_1(\cdot, \cdot)$ as
\begin{equation}\label{eqn:weak-form-derivation-6}
\begin{aligned}
a_1(u_1,v_1)
&:=
\left(\nabla u_1 ,{\bf K}_1 \nabla v_1\right)
+
\left(\beta_1\cdot \nabla u_1 f_1 , v_1\right)\\
&\hspace*{.25in}+
\left( \phi_1 u_1, v_1 + \theta h \nu\beta_1\cdot \nabla v_1\right) \\
&\hspace*{.25in}- \varepsilon\left(\nabla u_1, v_1+ \theta h \nu\beta_1\cdot \nabla v_1\right)_{\partial \Omega }
\end{aligned}
\end{equation}
which represents the combination of \eqref{eqn:weak-form-derivation-4} and \eqref{eqn:weak-form-derivation-5}.
Using the linear and bilinear forms defined above, equation \eqref{eqn:weak-form-derivation-1} becomes
\begin{align*} 
a_1(u_1,v_1) - b_1(u_2,v_1) 
 &= l_1(v_1), \quad \forall v_1 \in V_D(h).
\end{align*}
This provides the weak form for the first equation in \eqref{eikonal-system-w-linearization}.
By following the same derivation and introducing analogous bilinear and linear forms we arrive at the weak form of  \eqref{eikonal-system-w-linearization}:
find $(u_1, u_2) \in (V_D, V_G)$ such that
\begin{equation}\label{eqn:coupled-weak-form}
\begin{aligned}
	a_1(u_1,v_1) - b_1(u_2,v_1)
 &= l_1(v_1), \quad \forall v_1 \in V_D(h) \\
 a_2(u_2,v_2) - b_2(u_1,v_2)
 &= l_2(v_2), \quad \forall v_2 \in V_G(h). 
\end{aligned}
\end{equation}
The system \eqref{eqn:coupled-weak-form} then represents our finite element formulation for a given $\varepsilon$.
This results in a single linear system which solves for the epsilon dependent $u_1$ and $u_2$ simultaneously using the algebraic multilevel solver Multigraph 2.1 (cf. \cite{Bank2, Bank3}).
We remark that this is in contrast to the value iteration approach in \cite{Cornell1} which iterates between each equation in the coupled system.
In our setting, the only iteration is due to decreasing $\varepsilon$.

\section{Numerical Example}
We consider the road scenario depicted in Figure \ref{fig:road-scene}.
Here, a pair of disabled vehicles are present on the shoulder of the road due to a collision.
\begin{figure}
    \centering
    \includegraphics[scale=0.35]{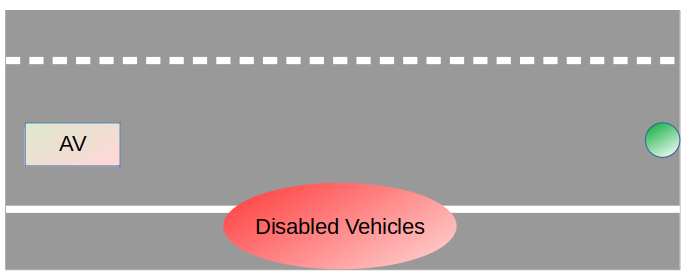}
    \caption{Road scenario} 
    \label{fig:road-scene}
    \vspace{-10pt}
\end{figure}
\begin{figure}
    \centering
    \includegraphics[scale=0.38]{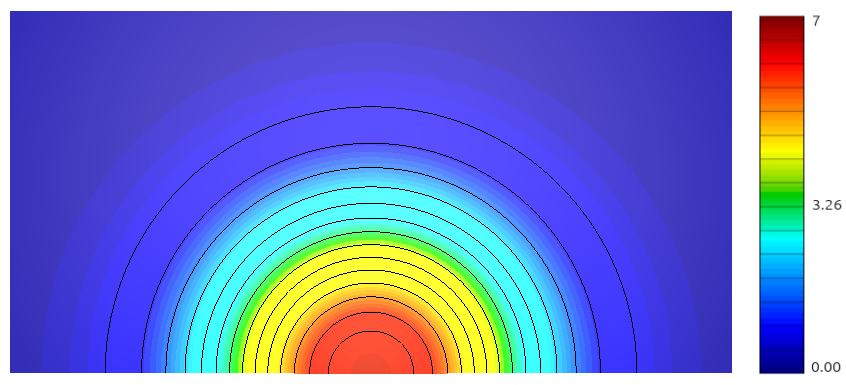}
    \caption{Counter plot of the spatially dependent likelihood of breakdown function $\phi(x)$ associated with the road layout in Figure \ref{fig:road-scene}.} 
    \label{fig:gaussian}
    \vspace{-10pt}
\end{figure}
The assumption is that a breakdown is more likely in regions near the pair of disabled vehicles due to potential debris etc. which the AV must account for in navigating to the goal location indicated by the green dot;
in this case, we assume the goal location and repair depot coincide. 
This increased likelihood of a breakdown is represented in the model by the function $\phi(x)$ given in Figure \ref{fig:gaussian}.
For the above scenario we define  $\Omega = [0,2] \times [0,1]$, let $G = D = (1.9, 0.5)$, and uniformly refine a total of seven times.
For the remaining parameters, we follow \cite{Cornell1} and set $f_1 =1$, $f_2 = 0.2$,  $R = 1$, $K_1=1$, $K_2 = 1$ and define $\phi(x) = 7 e^{-5(x-1)^2 - 5y^2}$ which is plotted above.
Additionally, we take $\lambda = 1$ and  $\phi_2 (x) = 3$.
Figures \ref{fig:u1} and \ref{fig:u2} display the contour plots associated with the value functions which correspond to vehicle operation in modes 1 and 2, respectively.
A trajectory from any starting point $(x_0, y_0)$ is obtained by moving in the direction of the gradient.
The contours show that the optimal trajectories move away from the disabled vehicles as expected.
\begin{figure}
    \centering
    \includegraphics[scale=0.35]{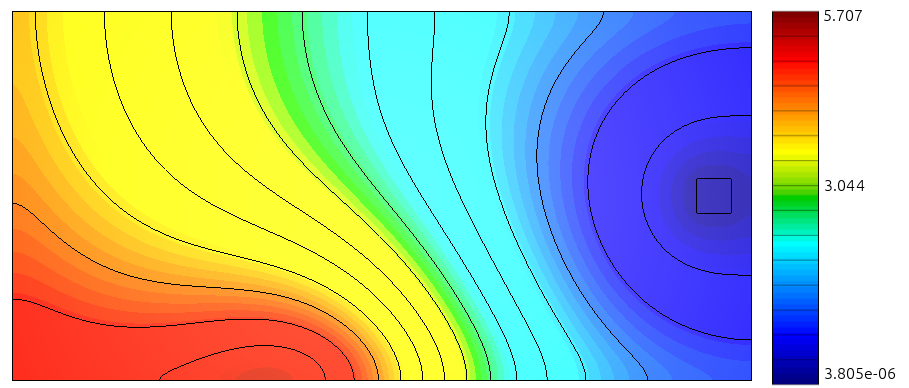}
    \caption{Value function $u_1$  corresponding to normal operation.} 
    \label{fig:u1}
    \vspace{-10pt}
\end{figure}
\begin{figure}
    \centering
    \includegraphics[scale=0.36]{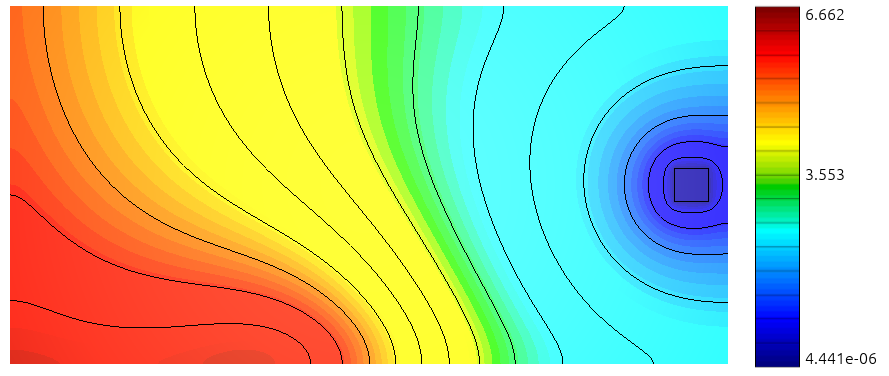}
    \caption{Value function $u_2$ which corresponds to operation after a partial breakdown.} 
    \label{fig:u2}
    \vspace{-10pt}
\end{figure}

\section{Conclusions and future work}
In this work, we considered the model of optimal path planning with random
breakdowns introduced in \cite{Cornell1} and analyzed from a theoretical point of view. We proved a comparison principle for viscosity solutions of the weakly coupled system of eikonal equations of the value functions. In particular, we described how to bypass the non-trivial degenerate coupling condition \eqref{coup}. Then we showed an example of how the lack of boundary conditions on the Aubry set can compromise the uniqueness of the solution.  Finally, in the same example, we have provided boundary conditions for which the system does not have solutions. It still remains an open question how to choose boundary conditions to ensure the existence of a solution. Indeed, the few existence results proved for weakly coupled systems require that the coupling matrix is non-degenerate \cite{Engler}, or that the Aubry set is empty \cite{Camilli}, or that the domain is a torus \cite{Camilli1}.

Future direction may include the extension of the model in \cite{Cornell1} to domains topologically more challenging, like networks, assuming not only that the vehicle is subject to breakdowns but also that the speed varies because of heterogeneous road conditions or the presence of lights at the nodes.  It would be also interesting to revisit the model introduced in \cite{Cornell1}  incorporating realistic features such as the uncertainty of road conditions and analysis of sideslip angles in the spirit of recent literature \cite{SELMANAJ20171, 6579735, 9400377}. 
Numerically, this work took advantage of finite element schemes used in convection-diffusion problems by iteratively solving a linearized form of the system augmented with artificial viscosity \eqref{eikonal-system-w-linearization}; specifically, the streamline diffusion approximation was used as a means of obtaining stable numerical solutions at each iteration.
A number of schemes exist for such problems and the exploration of their utility in the context of the coupled system considered in this paper is one potential direction for future work.

%\begin{thebibliography}{00}
\bibliographystyle{IEEEtran}
\bibliography{references}

\begin{thebibliography}{10}
\providecommand{\url}[1]{#1}
\csname url@samestyle\endcsname
\providecommand{\newblock}{\relax}
\providecommand{\bibinfo}[2]{#2}
\providecommand{\BIBentrySTDinterwordspacing}{\spaceskip=0pt\relax}
\providecommand{\BIBentryALTinterwordstretchfactor}{4}
\providecommand{\BIBentryALTinterwordspacing}{\spaceskip=\fontdimen2\font plus
\BIBentryALTinterwordstretchfactor\fontdimen3\font minus
  \fontdimen4\font\relax}
\providecommand{\BIBforeignlanguage}[2]{{%
\expandafter\ifx\csname l@#1\endcsname\relax
\typeout{** WARNING: IEEEtran.bst: No hyphenation pattern has been}%
\typeout{** loaded for the language `#1'. Using the pattern for}%
\typeout{** the default language instead.}%
\else
\language=\csname l@#1\endcsname
\fi
#2}}
\providecommand{\BIBdecl}{\relax}
\BIBdecl

\bibitem{Reeds}
J.~Reeds and L.~Shepp, ``Optimal paths for a car that goes both forwards and
  backwards,'' \emph{Pacific journal of mathematics}, vol. 145, no.~2, pp.
  367--393, 1990.

\bibitem{Alt}
K.~Alton and I.~M. Mitchell, ``Optimal path planning under different norms in
  continuous state spaces,'' in \emph{Proceedings 2006 IEEE International
  Conference on Robotics and Automation, 2006. ICRA 2006.}\hskip 1em plus 0.5em
  minus 0.4em\relax IEEE, 2006, pp. 866--872.

\bibitem{Kav}
L.~E. Kavraki, P.~Svestka, J.-C. Latombe, and M.~H. Overmars, ``Probabilistic
  roadmaps for path planning in high-dimensional configuration spaces,''
  \emph{IEEE Transactions on Robotics and Automation}, vol.~12, no.~4, pp.
  566--580, 1996.

\bibitem{Lav}
S.~LaValle, ``Rapidly-exploring random trees: A new tool for path planning,''
  \emph{Research Report 9811}, 1998.

\bibitem{SELMANAJ20171}
\BIBentryALTinterwordspacing
D.~Selmanaj, M.~Corno, G.~Panzani, and S.~M. Savaresi, ``Vehicle sideslip
  estimation: A kinematic based approach,'' \emph{Control Engineering
  Practice}, vol.~67, pp. 1--12, 2017. [Online]. Available:
  \url{https://www.sciencedirect.com/science/article/pii/S0967066117301491}
\BIBentrySTDinterwordspacing

\bibitem{6579735}
J.-H. Yoon and H.~Peng, ``Robust vehicle sideslip angle estimation through a
  disturbance rejection filter that integrates a magnetometer with gps,''
  \emph{IEEE Transactions on Intelligent Transportation Systems}, vol.~15,
  no.~1, pp. 191--204, 2014.

\bibitem{9400377}
X.~Xia, E.~Hashemi, L.~Xiong, A.~Khajepour, and N.~Xu, ``Autonomous vehicles
  sideslip angle estimation: Single antenna gnss/imu fusion with observability
  analysis,'' \emph{IEEE Internet of Things Journal}, vol.~8, no.~19, pp.
  14\,845--14\,859, 2021.

\bibitem{Cornell1}
M.~Gee and A.~Vladimirsky, ``Optimal path-planning with random breakdowns,''
  \emph{IEEE Control Systems Letters}, vol.~6, pp. 1658--1663, 2021.

\bibitem{Dav}
H.~T. Davis, \emph{Introduction to nonlinear differential and integral
  equations}.\hskip 1em plus 0.5em minus 0.4em\relax US Atomic Energy
  Commission, 1960.

\bibitem{Flem}
W.~H. Fleming and H.~M. Soner, \emph{Controlled Markov processes and viscosity
  solutions}.\hskip 1em plus 0.5em minus 0.4em\relax Springer Science \&
  Business Media, 2006, vol.~25.

\bibitem{Prot}
M.~H. Protter and H.~F. Weinberger, \emph{Maximum principles in differential
  equations}.\hskip 1em plus 0.5em minus 0.4em\relax Prentice Hall, Englewood
  Cliffs, 1967.

\bibitem{Camilli1}
F.~Camilli, O.~Ley, P.~Loreti, and V.~D. Nguyen, ``Large time behavior of
  weakly coupled systems of first-order hamilton--jacobi equations,''
  \emph{Nonlinear Differential Equations and Applications NoDEA}, vol.~19, pp.
  719--749, 2012.

\bibitem{Cagn}
F.~Cagnetti, D.~Gomes, and H.~V. Tran, ``Adjoint methods for obstacle problems
  and weakly coupled systems of pde,'' \emph{ESAIM: Control, Optimisation and
  Calculus of Variations}, vol.~19, no.~3, pp. 754--779, 2013.

\bibitem{Mitake1}
H.~Mitake, A.~Siconolfi, H.~V. Tran, and N.~Yamada, ``A lagrangian approach to
  weakly coupled hamilton--jacobi systems,'' \emph{SIAM Journal on Mathematical
  Analysis}, vol.~48, no.~2, pp. 821--846, 2016.

\bibitem{Davini1}
A.~Davini and M.~Zavidovique, ``Aubry sets for weakly coupled systems of
  hamilton--jacobi equations,'' \emph{SIAM Journal on Mathematical Analysis},
  vol.~46, no.~5, pp. 3361--3389, 2014.

\bibitem{Engler}
H.~Engler and S.~M. Lenhart, ``Viscosity solutions for weakly coupled systems
  of hamilton-jacobi equations,'' \emph{Proceedings of the London Mathematical
  Society}, vol.~3, no.~1, pp. 212--240, 1991.

\bibitem{Ishi}
H.~Ishii and S.~Koike, ``Viscosity solutions for monotone systems of
  second--order elliptic pdes,'' \emph{Communications in partial differential
  equations}, vol.~16, no. 6-7, pp. 1095--1128, 1991.

\bibitem{Sethian1}
J.~A. Sethian, ``A fast marching level set method for monotonically advancing
  fronts.'' \emph{proceedings of the National Academy of Sciences}, vol.~93,
  no.~4, pp. 1591--1595, 1996.

\bibitem{Kao}
C.-Y. Kao, S.~Osher, and Y.-H. Tsai, ``Fast sweeping methods for static
  hamilton--jacobi equations,'' \emph{SIAM journal on numerical analysis},
  vol.~42, no.~6, pp. 2612--2632, 2005.

\bibitem{10342800}
C.~Parkinson and K.~Polage, ``An efficient semi-real-time algorithm for path
  planning in the hamilton–jacobi formulation,'' \emph{IEEE Control Systems
  Letters}, vol.~7, pp. 3621--3626, 2023.

\bibitem{parkinson2023efficient}
C.~Parkinson and I.~Boyle, ``Efficient and scalable path-planning algorithms
  for curvature constrained motion in the hamilton-jacobi formulation,'' 2023.

\bibitem{Caboussat}
A.~Caboussat, R.~Glowinski, and T.-W. Pan, ``On the numerical solution of some
  eikonal equations: An elliptic solver approach,'' \emph{Chinese Annals of
  Mathematics, Series B}, vol.~36, no.~5, pp. 689--702, 2015.

\bibitem{Yang}
Y.~Yang, W.~Hao, and Y.-T. Zhang, ``A continuous finite element method with
  homotopy vanishing viscosity for solving the static eikonal equation,''
  \emph{Communications in Computational Physics}, vol.~31, no.~5, pp.
  1402--1433, 2022.

\bibitem{guermond2008fast}
J.-L. Guermond, F.~Marpeau, and B.~Popov, ``A fast algorithm for solving
  first-order pdes by l1-minimization,'' \emph{Communications in Mathematical
  Sciences}, vol.~6, no.~1, pp. 199--216, 2008.

\bibitem{Flad}
D.~Flad, A.~Pradhan, and S.~Murman, ``Arbitrary order solutions for the eikonal
  equation using a discontinuous galerkin method,'' \emph{arXiv preprint
  arXiv:2108.05950}, 2021.

\bibitem{Morton}
K.~W. Morton, \emph{Revival: Numerical solution of convection-diffusion
  problems (1996)}.\hskip 1em plus 0.5em minus 0.4em\relax CRC Press, 2019.

\bibitem{Cartee}
E.~Cartee, A.~Farah, A.~Nellis, J.~Van~Hook, and A.~Vladimirsky, ``Quantifying
  and managing uncertainty in piecewise-deterministic markov processes,''
  \emph{SIAM/ASA Journal on Uncertainty Quantification}, vol.~11, no.~3, pp.
  814--847, 2023.

\bibitem{Visc}
M.~G. Crandall, L.~C. Evans, and P.-L. Lions, ``Some properties of viscosity
  solutions of hamilton-jacobi equations,'' \emph{Transactions of the American
  Mathematical Society}, vol. 282, no.~2, pp. 487--502, 1984.

\bibitem{Roquejoffre}
\BIBentryALTinterwordspacing
G.~Namah and J.-M. Roquejoffre, ``Remarks on the long time behavior of the
  solutions of hamilton-jacobi equations,'' \emph{Communications in Partial
  Differential Equations}, vol.~24, no. 5-6, pp. 883--893, 1999. [Online].
  Available: \url{https://doi.org/10.1080/03605309908821451}
\BIBentrySTDinterwordspacing

\bibitem{Brooks1}
\BIBentryALTinterwordspacing
A.~N. Brooks and T.~J.~R. Hughes, ``Streamline upwind/{P}etrov-{G}alerkin
  formulations for convection dominated flows with particular emphasis on the
  incompressible {N}avier-{S}tokes equations,'' \emph{Comput. Methods Appl.
  Mech. Engrg.}, vol.~32, no. 1-3, pp. 199--259, 1982, fENOMECH ''81, Part I
  (Stuttgart, 1981). [Online]. Available:
  \url{https://doi.org/10.1016/0045-7825(82)90071-8}
\BIBentrySTDinterwordspacing

\bibitem{Johnson}
C.~Johnson, \emph{Numerical solution of partial differential equations by the
  finite element method}.\hskip 1em plus 0.5em minus 0.4em\relax Cambridge
  University Press, Cambridge, 1987.

\bibitem{Bank1}
R.~E. Bank, P.~S. Vassilevski, and L.~T. Zikatanov, ``Arbitrary dimension
  convection--diffusion schemes for space-time discretizations,'' \emph{Journal
  of Computational and Applied Mathematics}, vol. 310, pp. 19--31, 2017.

\bibitem{Bank2}
R.~E. Bank and R.~K. Smith, ``An algebraic multilevel multigraph algorithm,''
  \emph{SIAM Journal on Scientific Computing}, vol.~23, no.~5, pp. 1572--1592,
  2002.

\bibitem{Bank3}
\BIBentryALTinterwordspacing
R.~E. Bank. (2015) Multigraph 2.1. [Online]. Available:
  \url{https://ccom.ucsd.edu/\~{}reb/software.html}
\BIBentrySTDinterwordspacing

\bibitem{Camilli}
\BIBentryALTinterwordspacing
F.~Camilli and P.~Loreti, ``{Comparison results for a class of weakly coupled
  systems of eikonal equations},'' \emph{Hokkaido Mathematical Journal},
  vol.~37, no.~2, pp. 349 -- 362, 2008. [Online]. Available:
  \url{https://doi.org/10.14492/hokmj/1253539559}
\BIBentrySTDinterwordspacing

\end{thebibliography}
%\end{thebibliography}

\end{document}